\documentclass[12pt,a4paper]{amsart}%
\usepackage{imakeidx}
\usepackage{amsmath}
\usepackage{fullpage}
\usepackage{comment}
\usepackage{MyCommA}
\usepackage[final]{pdfpages}
\usepackage{tikz-cd}
\makeindex
\usetikzlibrary{calc}

\newcommand{\Dima}[1]{{{#1}}}
\newcommand{\Rami}[1]{{{#1}}}
\newcommand{\RamiA}[1]{{{#1}}}
\newcommand{\RamiB}[1]{{{#1}}}
\newcommand{\RamiC}[1]{{{#1}}}
\newcommand{\RamiD}[1]{{{#1}}}

\newcommand{\Eitan}[1]{{{#1}}}

\newcommand{\NextVer}[1]{}
\newcommand{\sub}{\subset}
\newcommand{\noleft}{\left.\kern-\nulldelimiterspace}

\newtheorem{innercustomthm}{Theorem}

\begin{document}
	
	\author[Aizenbud]{Avraham Aizenbud}
	\address{Avraham Aizenbud,
		Faculty of Mathematical Sciences,
		Weizmann Institute of Science,
		76100
		Rehovot, Israel}
	\email{aizenr@gmail.com}
	\urladdr{https://www.wisdom.weizmann.ac.il/~aizenr/}
	
	\author[Gourevitch]{Dmitry Gourevitch}
	\address{Dmitry Gourevitch,
		Faculty of Mathematical Sciences,
		Weizmann Institute of Science,
		76100
		Rehovot, Israel}
	\email{dimagur@weizmann.ac.il}
	\urladdr{https://www.wisdom.weizmann.ac.il/~dimagur/}

\author[Kazhdan]{David Kazhdan}
	\address{David Kazhdan,
    Einstein Institute of Mathematics, Edmond J. Safra Campus, Givaat Ram The
Hebrew University of Jerusalem, Jerusalem, 91904, Israel}
	\email{david.kazhdan@mail.huji.ac.il}
	\urladdr{https://math.huji.ac.il/~kazhdan/}
	
\author[Sayag]{Eitan Sayag}
	\address{Eitan Sayag,
    Department of Mathematics, Ben Gurion University of the Negev, P.O.B. 653,
Be’er Sheva 84105, ISRAEL}
	\email{eitan.sayag@gmail.com}
	\urladdr{www.math.bgu.ac.il/~sayage}

	\date{\today}


	%
	%
	%
	%
	%
	%
	%
	%
	

\title{Invertible top form on the Hilbert scheme of a plane in positive characteristic}
\maketitle
\begin{abstract} 
We prove that the Hilbert scheme of the plane in positive characteristic admits an invertible
top differential form.

This implies certain integrability properties 
of the symmetric powers of the plane. This allows to define a function on the collection of 
\Rami{monic}
polynomials over a local field which can be thought of as a variant of the inverse square root of the discriminant. In characteristic $0$ it essentially coincides with this inverse square root, however in general it is quite different, \Eitan{and} unlike this inverse square root, it is locally summable.
In \Rami{a sequel work} \cite{AGKS3} we use this local summability in order to prove the positive characteristic analog of Harish-Chandra's \Eitan{local} integrability theorem \Eitan{of characters of representations} under certain conditions. 

The main results of this paper are known in characteristic zero. In fact a stronger result is known: there is a symplectic form on the Hilbert scheme of a plane. 

\RamiD{After the completion of the work on this text, it was pointed out to us that its main results are known. See \S1.4 for more details.}
\end{abstract}

\tableofcontents 
\section{Introduction}

Throughout the paper we fix a  field \tdef{$F$} 
of arbitrary characteristic. 
We will also fix a natural number \tdef{$n$}.

\subsection{The Hilbert Scheme}
In order to formulate our results let us first recall the definition of the Hilbert scheme.
\begin{definition}\label{def:Hscheme}
    Let $Sch_F$ be the category of $F$-Schemes.
    For an $F$-algebraic variety $\bfZ$ define the Hilbert functor 
    $Hilb_n(\bfZ): Sch_F^{op}\to sets$  by 
   \begin{multline*}
       \mdef{Hilb_n(\bfZ)(\bfS)}:=\\ 
   \{\text{\RamiD{closed} sub-scheme } \bfY\sub \bfS\times \Rami{\bfZ}\, \vert \, (pr_{\bfS})_*(\cO_{\bfY})
   \text{ is locally free of rank }n\text{ over }\bfS,\}
   \end{multline*} 
   where $pr_{\bfS}$ is the projection.
\end{definition}

\begin{theorem}[{\cite{Gro62}, see also \cite[Theorem 7.2.3]{BK05}}]\label{thm:HScheme}
If $\bfZ$ is a quasi-projective variety then 
    the Hilbert functor $Hilb_n(\bfZ)$ is representable by a scheme which we denote by $\mdef{\bfZ^{[n]}}$. 
\end{theorem}
\begin{theorem}[{See {\it e.g.} \cite[Theorem 7.4.1]{BK05}}]\label{thm:dim.hilb}
If $\bfZ$ is a smooth  quasi-projective \RamiD{absolutely} irreducible algebraic surface then 
    $\bfZ^{[n]}$ is a smooth irreducible variety of dimension $2n$.
\end{theorem}
\subsection{Main results}
We prove
the following:
\begin{introtheorem}\label{thm:main}
There exists an invertible top differential form on     $(\bA^2)^{[n]}$. 
\end{introtheorem}

\subsection{Relation to the singularities of the symmetric power of the plane}
\Cref{thm:main} is related to the singularities of  the symmetric power of the plane. In order to formulate \Eitan{this relation} we introduce some notations:

\begin{definition}
    For a quasi projective algebraic variety $\bfZ$ 
    define its symmetric power by:   
    $$\mdef{\bfZ^{(n)}}:=\bfZ^n//S_n.$$
    Here $//$ denotes the categorical quotient. By \RamiC{\Cref{cor:fac}} below this quotient exists.
\end{definition}
\begin{notation}
Let $\bfZ$ be a quasi-projective variety. Let $x\in \bfZ^{[n]}(\bar F)$. It corresponds to a sheaf of ideals $\cI_x \subset \Dima{\cO}_{\bfZ_{\bar F}}$. For any $z\in \bfZ(\bar F)$ denote $$\mdef{n_x(z)}:=\dim (\Dima{\cO}_{\bfZ_{\bar F},z}/(\cI_x)_{z}).$$   
\RamiC{This gives}
a multiset in $\bfZ(\bar F)$ of size $n$. By \Cref{lem:GamBij} below, we can interpret this multiset as a point in $\bfZ^{(n)}(\bar F)$. Denote this point by    
    $\mdef{\mathfrak{H}_{\bfZ,n}(x)}$.
\end{notation}
\begin{theorem}[{\cite[II.2,II.3]{Ive}, \cite[Theorem 7.3.1]{BK05}\footnote{The theorem is formulated in \cite{BK05} for algebraically closed fields. However, it is based on results from \cite{Ive} which do not make this assumption, so the proof is valid for any field.}}]\label{thm:charact-HilChow}
    Let $\bfZ$ be a quasi-projective variety. There exists (and unique) a projective morphism $\mdef{\mathfrak{H}_{\bfZ,n}}:\bfZ^{[n]}\to \bfZ^{(n)}$ that gives on the level of $\bar F$ points the map $\mathfrak{H}_{\bfZ,n}$ defined above. 

    This morphism is called the \tdef{Hilbert-Chow morphism}.
\end{theorem}
Theorems \ref{thm:dim.hilb} and \ref{thm:charact-HilChow} imply:
\begin{cor}\label{cor:Hil.chow.res}
Let $\bfZ$ be a (quasi-projective) smooth surface. Then 
    the Hilbert-Chow map ${\mathfrak{H}_{\bfZ,n}}:\bfZ^{[n]}\to \bfZ^{(n)}$ is a resolution of singularities. 
\end{cor}
\Cref{thm:main} is related to the properties of this resolution. In order to formulate these relations we make:
\begin{defn}
$\,$
\begin{enumerate}[(i)]
\item We recall that a \tdef{modification} $\gamma:\tilde\bfV\to \bfV$ of algebraic varieties is a birational proper morphism.
\item 
We call a modification $\gamma:\tilde\bfV\to \bfV$ of algebraic varieties \tdef[integrable modification]{integrable} if for any open $\bfU\sub \bfV$ and any top-form $\omega$ on the smooth locus of $\bfU$ the \RamiA{rational} form $\gamma^*(\omega)$ \RamiA{on $\gamma^{-1}(\bfU)$ is} regular on the smooth locus of $\gamma^{-1}(\bfU)$. 
\item We call such modification \tdef[sharply integrable modification]{sharply integrable} if $\gamma^*(\omega)$ vanishes only on $\gamma^{-1}(\bar \bfD)$ where $\bfD$ is the zero locus of $\omega$.
\item We call a variety \tdef[sharply integrable variety, integrable variety]{(sharply) integrable} if it admits a (sharply) integrable resolution of singularities.
\end{enumerate}    
\end{defn}
\begin{remark}\label{rem:rat.sing}
    In characteristic zero, one can show that TFAE:\footnote{See e.g.  \cite[{Appendix B,} Proposition 6.2]{AA})}
    \begin{enumerate}[(a)]
        \item the singularities of $\bfZ$ are rational,
        \item\label{it:intro.2} $\bfZ$ is integrable and Cohen-Macaulay.
    \end{enumerate}
    In positive characteristic there is no single accepted definition of rational singularities, and one can take condition \eqref{it:intro.2} as a definition.    
\end{remark}

We will see that \Cref{thm:main} follows from:
\begin{introtheorem}
\label{thm:int}
    The Hilbert-Chow map $$\mathfrak H_{\bA^2,n}:(\bA^2)^{[n]}\to (\bA^2)^{(n)}$$ is a sharply integrable modification.
\end{introtheorem}
This implies:
\begin{introcorollary}\label{cor:int}
    $(\bA^2)^{(n)}$ is sharply integrable.
\end{introcorollary}
\begin{remark}\label{rem:2dir}
    In fact, it is easy to see that     
    \Cref{thm:int} and \Cref{thm:main} are equivalent. So one could instead prove directly \Cref{thm:main} and deduce \Cref{thm:int}.
\end{remark}

\subsection{Additional remark}
\RamiD{
This text was written as a supporting paper for \cite{AGKS3}.
After the completion of the work on this text, it was pointed out to us that its main results are known. Specifically \Cref{thm:main} follows immediately from \cite[Theorem 7.4.6]{BK05}.
However, this paper contains some technical intermediate results, which are mostly standard but it is not easy to find reference to some of them, so we leave this text here for expository purposes.}
\subsection{Background and motivation}
\subsubsection{The characteristic zero case}
The characteristic zero  counterpart of the main results of  this paper is well known. In fact, stronger results are known. 
Namely, the Hilbert-Chow map for smooth surfaces in characteristic $0$ is a symplectic resolution \Rami{(see {\it e.g.} \cite[Theorem 1.17]{Nak})}.
This implies the characteristic zero  counterpart of \Cref{thm:int}. This also implies that 
 in characteristic zero the Hilbert scheme of the plane is a symplectic variety. This in turn implies the characteristic $0$  counterparts of \Cref{thm:main} and \Cref{cor:int}.  In addition, $(\bA^2)^{(n)}$ is a quotient of algebraic variety by a finite group, therefore, by \cite[Corollaire]{Boutot}, in characteristic zero its singularities are rational. As mentioned in \Cref{rem:rat.sing} this is equivalent to the fact that it is integrable and Cohen-Macaulay. 
\subsubsection{Relation to local finiteness of measures}
If the field $F$ is local, integrability of an algebraic variety $\bfZ$ implies that given a top form $\omega$ on its smooth locus, the corresponding measure  $|\omega|$ on $\bfZ(F)$ is locally finite.

Therefore, given a (locally) dominant map $\phi:\bfZ\to \bfY$ to  a smooth variety and a function $f\in C_c^\infty(\bfZ(F))$, the measure $\phi_*(f|\omega|)$ is also locally finite. Since it is also absolutely continuous (w.r.t. a smooth invertible measure on $\bfY(F)$), this measure has a locally summable density function.

Applying this consideration to the map $(\bA^2)^{(n)}\to (\bA^1)^{(n)}$ \RamiC{(induced by the projection $\bA^2\to \bA^1$)}  we get a \RamiC{locally summable} density function $\eta$ (defined up to multiplication by a smooth compactly supported function) on $(\bA^1)^{(n)}(F)$. Note that $(\bA^1)^{(n)}$ is naturally identified with the space of monic polynomials of degree $n$.

Over $\C$ this function is the absolute value of the inverse square root of the discriminant -- $|\Delta|^{-\frac 12}$.
Over a general  local field of characteristic zero, this function is bounded from above and from below by a constant times $|\Delta|^{-\frac 12}$. However, in positive characteristic this is no longer true. Moreover, in small positive characteristic the function $|\Delta|^{-\frac 12}$ is not locally summable. One can consider $\eta$ as a better behaved version of $|\Delta|^{-\frac 12}$. In \Rami{a sequel work} \cite{AGKS3} we use the local summability of $\eta$ in order to prove the positive characteristic analog of Harish-Chandra's integrability theorem under certain conditions. It turns out that for the sake of this theorem $\eta$ actually plays the role of $|\Delta|^{-\frac 12}$.
\subsection{Idea of the proof}

We define a closed subset $$(\bA^2)^{(n)}_{diag}\subset (\bA^2)^{(n)}$$ which corresponds to the diagonal copy of $\bA^2$. 
We prove:
\begin{introlem}[For a precise formulation see \Cref{lem:n1n2}]\label{lem:n1n2Int}
    Outside $(\bA^2)^{(n)}_{diag}$ the Hilbert-Chow map looks (locally in the {\'e}tale topology) like a product of the Hilbert-Chow maps for smaller values of $n$. 
\end{introlem}

We use this Lemma together with the induction hypothesis in order to prove the theorem outside \Rami{$(\bA^2)^{(n)}_{diag}$}. Then we use the fact that the complement to $$\mathfrak H_{\bA^2,n}^{-1}((\bA^2)^{(n)}_{diag})$$ in $(\bA^2)^{[n]}$ is big in order to deduce the result. 

This strategy works only for $n>2$, for $n=2$ we prove the  theorem by an explicit computation.

\subsection{Structure of the paper}
In \S \ref{sec:conv} we fix conventions that will be used throughout the paper.

In \S \ref{sec:fac} we study quotients of varieties by finite group actions.
In \S \ref{sec:Y}  we prove \Cref{thm:int}. In \S \ref{subsec:Pfn1n2} we prove \Cref{lem:n1n2Int}.

In \S \ref{sec:Pfmain} we prove \Cref{thm:main}.
\subsection{Acknowledgments}
 We wish to thank Pavel Etingof and Victor Ginzburg for discussing with us the situation with Hilbert scheme in positive characteristic.

During the preparation of this paper, A.A., D.G. and E.S. were partially supported by the ISF grant no. 1781/23. 
D.K. was partially supported by an ERC grant 101142781.

\section{Conventions}\label{sec:conv}
\begin{enumerate}[(a)]
    \item By a \tdef{variety} we mean a reduced scheme of finite type over $F$. 
    \item When we consider a fiber product of varieties, we always consider it in the category of schemes. \Rami{We use set-theoretical notations to define subschemes, whenever no ambiguity is possible.} 
    \item We will usually denote algebraic varieties by bold face letters (such as $\bfX$) and the spaces of their $F$-points by the corresponding usual face letters (such as $X:=\bfX(F)$). We use the same conventions when we want to interpret vector spaces as algebraic varieties.
    \item We will use the same letter to denote a morphism between algebraic varieties and the corresponding map between the sets of their $F$-points.
    \item We will use the symbol \tdef{$\square$} in a middle of a square diagram  in order to indicate that the square is Cartesian. 
    \item We will use numbers in a middle of a square diagram in order to refer to the square by the corresponding number.
    \item A \tdef{big open set} of an algebraic variety $\bfZ$  is an open set whose complement is of co-dimension at least 2 (in each component).    
    \item For a  variety $\bfZ$ we denote its smooth locus  by $\bfZ^{sm}$.
    \item For a smooth variety $\bfZ$ we denote by $\mdef{\Omega^{top}}(\bfZ)$ the sheaf of top differential forms on $\bfZ$.
    \item \Rami{For a variety $\bfZ$ and  a field  extension $E/F$, denote by \tdef{$\bfZ_E$} the extension of scalars to $E$. We use similar notation for morphisms.}
\end{enumerate}

\section{Factorizable actions}\label{sec:fac}
In this section we give some standard facts about quotients of an algebraic variety by a finite group which are slightly less standard in positive characteristic.

\begin{definition}
    Let a finite group $\Gamma$ act on a variety $\bfZ$. We say that this action is \tdef[factorizable action]{factorizable} if the categorical quotient \tdef{$\bfZ//\Gamma$} exists (as a variety), and the map $\bfZ \to \bfZ//\Gamma$ is finite. 
\end{definition}

\begin{lemma}\label{lem:GamBij}
    Let a finite group $\Gamma$ act factorizably on a variety $\bfZ$. Then
the map $\gamma:\bfZ(\RamiC{\bar F})/\Gamma \to (\bfZ//\Gamma)(\RamiC{\bar F})$ is a bijection, where $\bfZ(\RamiC{\bar F})/\Gamma$ denotes the set of $\Gamma$-orbits in $\bfZ(\RamiC{\bar F})$. 
\end{lemma}
\begin{proof}
    Since the map is affine, we can assume that $\RamiC{\bfZ}$ is affine. 
    The map $\gamma$ is onto by the going up theorem. To show that it is one-to-one it is enough to show that for every $O_1,O_2\in \RamiC{\bfZ}(\RamiC{\bar F})/\Gamma$ there exists $f\in (\RamiC{\bar F}[\RamiC{\bfZ}])^\Gamma$ such that $f|_{O_1}=0$ and $f|_{O_2}=1$.
    Let $f'\in \RamiC{\bar F}[\RamiC{\bfZ}]$ such that $f|_{O_1}=0$ and $f|_{O_2}=1$, and let $f:=\prod_{g\in \Gamma}g^*(f')$.
\end{proof}
\subsection{\Rami{Factorizabilty of \RamiC{quasi-projective} varieties and compatibility with open embeddings}}
\Rami{
In this subsection we prove that \RamiC{quasi-projective} varieties are factorizable (see \RamiC{\Cref{cor:fac}}) and the quotients of factorizable varieties are compatible with open embeddings (see  \Cref{lem:GamU}).

The only place where the positivity of characteristic presents an additional difficulty is the compatibility for the affine case. See the base of the induction in the proof of \Cref{lem:GamU0}. There we can not use the standard averaging method, and we use its multiplicative version instead.
}
\begin{prop}[{cf. \cite[Lec. 10, pp. 124-125]{Har} or \cite{Ser}}]\label{prop:aff.fac}
    Let a finite group $\Gamma$ act on an affine variety $\bfZ$. Then the action is factorizable and 
    $\bfZ//\Gamma\cong \Spec(\cO(\bfZ)^{\Gamma})$.
\end{prop}

\begin{lemma}\label{lem:glue}
\Rami{
    Let a finite group $\Gamma$ act  on a  variety $\bfZ$. Let $\bfZ=\bfZ_1\cup \bfZ_2$ be a cover of $\bfZ$ by open, $\Gamma$-invariant factorizable  sets. Assume that $\bfV:=\bfZ_1\cap \bfZ_2$  is also factorizable and we have Cartesian squares:
    }    
\begin{equation}\label{lem:glue:diag1}    
\begin{tikzcd}
\mathbf{V} \arrow[d] \arrow[r] \arrow[rd, phantom, "\square"] &
\mathbf{Z}_1 \arrow[d] \\
\mathbf{V}//\Gamma \arrow[r] &
\mathbf{Z}_1//\Gamma
\end{tikzcd}
\qquad
\begin{tikzcd}
\mathbf{V} \arrow[d] \arrow[r] \arrow[rd, phantom, "\square"] &
\mathbf{Z}_2 \arrow[d] \\
\mathbf{V}//\Gamma \arrow[r] &
\mathbf{Z}_2//\Gamma
\end{tikzcd}
\end{equation}
\Rami{
with the (lower) horizontal maps being open embeddings.
Let $$\bfW:=\bfZ_1//\Gamma \sqcup_{\bfV//\Gamma } \bfZ_2//\Gamma.$$ Then the natural map $\Dima{\bfZ}\to \bfW$ is the categorical quotient map and it is finite. Moreover, we have the following Cartesian squares:}
\begin{equation}\label{lem:glue:diag2}    
\begin{tikzcd}
\mathbf{Z}_1 \arrow[d] \arrow[r] \arrow[rd, phantom, "\square"] &
\mathbf{Z} \arrow[d] \\
\mathbf{Z}_1//\Gamma \arrow[r] &
\mathbf{Z}//\Gamma
\end{tikzcd}
\qquad
\begin{tikzcd}
\mathbf{Z}_2 \arrow[d] \arrow[r] \arrow[rd, phantom, "\square"] &
\mathbf{Z} \arrow[d] \\
\mathbf{Z}_2//\Gamma \arrow[r] &
\mathbf{Z}//\Gamma
\end{tikzcd}
\end{equation}
\end{lemma}
\begin{proof}
    \Rami{
Let $\gamma:\bfZ\to\bfA$ be a $\Gamma$-invariant map to an algebraic variety. The maps $\gamma|_{\bfZ_1}$, $\gamma|_{\bfZ_2}$, and $\gamma|_{\bfV}$  factor through maps $\alpha:\bfZ_2//\Gamma\to \bfA$,   $\beta:\bfZ_1//\Gamma\to \bfA$ and $\delta:\bfV//\Gamma\to \bfA$. These maps give a factorization of $\gamma$ via a map  $\bfW\to \bfA$. 
The uniqueness of such factorization is proven 
similarly (but simpler). Thus we have proven that the natural map $\bfZ\to \bfW$ is the categorical quotient.

Let us now show that  the diagrams \eqref{lem:glue:diag2} are Cartesian.  Let $\phi_\Gamma:\bfZ\to \bfZ//\Gamma\cong\bfW$ be the categorical quotient map. 
We need to show that $\phi_\Gamma^{-1}(\bfZ_1//\Gamma)=\bfZ_1$ (and similarly for $\bfZ_2$). Let $x\in \phi^{-1}_\Gamma(\bfZ_1//\Gamma)$. If $x\in \bfZ_1$ we are done. Otherwise  $x\in \bfZ_2$. This implies that $\phi_\Gamma(x)\in \bfV//\Gamma$. Thus, by  the right Cartesian square in \eqref{lem:glue:diag1} we get that $x\in \bfV$ and we are done.

Finally, the finiteness of $\phi_\Gamma$ follows from the Cartesian squares 
\eqref{lem:glue:diag2} and the finiteness of the maps $\bfZ_i\to \bfZ_i//\Gamma$. }
\end{proof}

\begin{lemma}\label{lem:GamU0}
    Let a finite group $\Gamma$ act  on a  variety $\bfZ$. 
    Let $\bfU\sub \bfZ$ be an open $\Gamma$-invariant set. 
    \Rami{Assume that $\bfZ$ can be covered  by open affine $\Gamma$-invariant sets.}
    Then
    \Rami{
    \begin{enumerate}[(i)]
        \item \label{lem:GamU:itZ} the action of $\Gamma$ on $\bfZ$ is factorizable.
        \item\label{lem:GamU:itii} The action of $\Gamma$ on $\bfU$ is factorizable.
        \item\label{lem:GamU0:itiii}
        The following natural diagram is a Cartesian square.
\begin{equation*}
        \xymatrix{
        \bfU \ar[d] \ar[r] &\bfZ\ar[d]\\
        \bfU//\Gamma \ar[r] &\bfZ//\Gamma}
        \end{equation*}
        \item \label{lem:GamU:itiv}
        The bottom arrow in the diagram is an open embedding. 
    \end{enumerate}
    }
\end{lemma}
\begin{proof}
\Rami{We prove the statement by induction on the size $N$ of the (minimal) cover of $\bfZ$ by open affine $\Gamma$-invariant sets.} 
\begin{itemize}
    \item[Base $N=1$:] 
    \Rami{\eqref{lem:GamU:itZ} follows from \Cref{prop:aff.fac}. This implies also that $\bfZ//\Gamma$ is affine.}

    Let $\bfA \subset \bfZ$ be the complement of $\bfU$.  
    For any closed point $x\in \bfU$, we can find a function $f_x\in \cO_\bfZ(\bfZ)$ s.t. $f_x(\bfA )=0$ and  $f_x(\Gamma \cdot x)=\{1\}$. 
    
    Let $$g_x=\prod_{\gamma\in\Gamma} \gamma^*(f_x).$$
    Let $\bfU_x\subset \bfZ$ be the non-vanishing locus of $g_x$.
    Note that \Dima{$g_x\in \cO_{\bfZ}(\bfZ)^\Gamma=\cO_{\bfZ//\Gamma}(\bfZ//\Gamma)$}. Let $\bfV_x\sub \bfZ//\Gamma$ be the non-vanishing locus of $g_x$ when considered as a function on  $\bfZ//\Gamma$. By \Cref{prop:aff.fac}, the action of $\Gamma$ on  $\bfU_x$ is factorizable. $\bfU_x//\Gamma\cong\bfV_x$. Let $$\bfV=\bigcup_{x\in\bfZ \text{ is closed }} \bfV_x$$
    It is easy to deduce that  $\bfV\cong \bfU//\Gamma$ and we have the required Cartesian square.
    \item[Step:] 
    \Rami{
    Write $\bfZ=\bigcup_{i=1}^N \bfZ_i$ where $\bfZ_i$ are open, affine, and $\Gamma$-invariant. Let $\bfY=\bigcup_{i=2}^N \bfZ_i$. 
    Let $\bfV:=\bfZ_1\cap \bfY$.
The previous lemma (\Cref{lem:glue}) and the induction hypothesis applied to the pairs 
 $\bfV \subset \bfZ_1$, $\bfV \subset \bfY$,
 $\bfV \cap \bfU \subset \bfZ_1 \cap \bfU$,  and $\bfV \cap \bfU \subset \bfY\cap \bfU$
imply \eqref{lem:GamU:itZ} and \eqref{lem:GamU:itii}.

We also get the following Cartesian squares:
}
\begin{equation}\label{lem:GamU0:sq1}    
\begin{tikzcd}
\bfZ_1 \arrow[d] \arrow[r] \arrow[rd, phantom, "\square"] &
\bfZ \arrow[d] \\
\bfZ_1//\Gamma \arrow[r] &
\bfZ//\Gamma
\end{tikzcd}
\qquad
\begin{tikzcd}
\bfY \arrow[d] \arrow[r] \arrow[rd, phantom, "\square"] &
\bfZ \arrow[d] \\
\bfY//\Gamma \arrow[r] &
\bfZ//\Gamma
\end{tikzcd}
\end{equation}

\begin{equation}
\label{lem:GamU0:sq2}
\begin{tikzcd}
\bfZ_1 \cap \bfU \arrow[d] \arrow[r] \arrow[rd, phantom, "\square"] &
\bfU \arrow[d] \\
(\bfZ_1\cap \bfU)//\Gamma \arrow[r] &
\bfU//\Gamma
\end{tikzcd}
\qquad
\begin{tikzcd}
\bfY \cap \bfU \arrow[d] \arrow[r] \arrow[rd, phantom, "\square"] &
\bfU \arrow[d] \\
(\bfY\cap \bfU)//\Gamma \arrow[r] &
\bfU//\Gamma
\end{tikzcd}
\end{equation}
\Rami{
with horizontal maps being open embeddings. Moreover, 
\begin{equation}\label{lem:GamU0:cov}    
\bfZ//\Gamma=\bfZ_1//\Gamma\cup \bfY//\Gamma \text{ and }
\bfU//\Gamma=(\bfZ_1 \cap \bfU)//\Gamma\cup(\bfY \cap \bfU)//\Gamma
\end{equation}
    Applying the induction hypothesis for the pairs $\bfZ_1\cap \bfU \subset \bfZ_1$ and $\bfY\cap \bfU \subset \bfY$ we obtain the following Cartesian squares:}
    \begin{equation}
\label{lem:GamU0:sq3}
\begin{tikzcd}
\bfZ_1 \cap \bfU \arrow[d] \arrow[r] \arrow[rd, phantom, "\square"] &
\bfZ_1 \arrow[d] \\
(\bfZ_1\cap \bfU)//\Gamma \arrow[r] &
\bfZ_1//\Gamma
\end{tikzcd}
\qquad
\begin{tikzcd}
\bfY \cap \bfU \arrow[d] \arrow[r] \arrow[rd, phantom, "\square"] &
\bfY \arrow[d] \\
(\bfY\cap \bfU)//\Gamma \arrow[r] &
\bfY//\Gamma
\end{tikzcd}
\end{equation}
\Rami{
with horizontal maps being open embeddings.

This together with \eqref{lem:GamU0:sq1} gives the following Cartesian squares:    
    }
        \begin{equation}
\label{lem:GamU0:sq4}
\begin{tikzcd}
\bfZ_1 \cap \bfU \arrow[d] \arrow[r] \arrow[rd, phantom, "\square"] &
\bfZ \arrow[d] \\
(\bfZ_1\cap \bfU)//\Gamma \arrow[r] &
\bfZ//\Gamma
\end{tikzcd}
\qquad
\begin{tikzcd}
\bfY \cap \bfU \arrow[d] \arrow[r] \arrow[rd, phantom, "\square"] &
\bfZ \arrow[d] \\
(\bfY\cap \bfU)//\Gamma \arrow[r] &
\bfZ//\Gamma
\end{tikzcd}
\end{equation}
\Rami{
with horizontal maps being open embeddings.

This together with \eqref{lem:GamU0:cov} proves \eqref{lem:GamU:itiv}. It remains to prove \eqref{lem:GamU0:itiii}. For this it is enough to show that $\phi^{-1}(\bfU//\Gamma)=\bfU,$ \Dima{where $\phi:\bfZ\to \bfZ//\Gamma$ denotes the quotient map.}
This follows from \eqref{lem:GamU0:sq2} and \eqref{lem:GamU0:cov}.
}   
\end{itemize}
\end{proof}
The  last Lemma gives us 2 corollaries:
\Rami{
\begin{cor}\label{lem:fac}
Let a finite group $\Gamma$ act on a variety $\bfZ$.    Then \RamiB{TFAE:}
    \begin{enumerate}[(i)]
        \item the action \RamiC{of $\Gamma$} is factorizable.
        \item $\bfZ$ can be covered  by open affine $\Gamma$-invariant sets.
    \end{enumerate}$ $
\end{cor}
}
\begin{cor}\label{lem:GamU}
Let a finite group $\Gamma$ act factorizably on a  variety $\bfZ$. 
    Let $\bfU\sub \bfZ$ be an open $\Gamma$-invariant set.     
    Then
    \Rami{
    \begin{enumerate}[(i)]
        \item the action of $\Gamma$ on $\bfU$ is factorizable.
        \item \label{eq:csg}
        The following natural diagram is a Cartesian square.
\begin{equation*}
        \xymatrix{
        \bfU \ar[d] \ar[r] &\bfZ\ar[d]\\
        \bfU//\Gamma \ar[r] &\bfZ//\Gamma}
        \end{equation*}
        \item 
        The bottom arrow in the diagram is an open embedding. 
    \end{enumerate}
    }
\end{cor}

\begin{lemma} \label{lem:q.proj.open}
    Let $\bfZ$ be a quasi-projective variety, and $\bfA\subset \bfZ$ be a finite subvariety. Then there exists an open affine $\bfV \sub \bfZ$ s.t.  $\bfA\sub \bfV$.
\end{lemma}
\begin{proof}
$ $
    \begin{enumerate}[{Case} 1.]
        \item\label{lem:q.proj.open:ca1} $\bfZ$ is the  projective space  and $F$ is infinite.\\
        \Rami{In this case one can take $\bfV$ to be a compliment to a hyperplane that does not intersect $\bfA$.
        }
        \item $\bfZ$ is the  projective space.\footnote{\Rami{In fact, an  accurate repetition of Case \ref{lem:q.proj.open:ca1} in the language of schemes will give a proof for this case, however we prefer the more geometric approach below.}}\\
        \Rami{By the previous case we may assume that $F$ is finite (and hence perfect). From the previous case we have \Dima{an open affine subset $\bfV'\subset \bfZ_{\bar F}$ that includes $\bfA_{\bar F}$}. We can find a finite (Galois) extension $E/F$ s.t. there exists $\bfV''\subset \bfZ_{E}$ satisfying $\bfV''_{\bar F}=\bfV'$. Now, we can find $\bfV$ s.t. $\bfV_{E}=\bigcup_{\alpha\in Gal(E/F)} \alpha(\bfV'')$. It is easy to see that $\bfV$ satisfies the requirements.
        }
        
        
        \item \label{q.proj.open:3}$\bfZ$ is a  projective variety.\\
        Follows from the previous case.
        \item $\bfZ$ is a  quasi-affine variety.\\
        Embed $\bfZ$ as an open subset of an affine variety $\bfZ'$. Let $\bfW$ be the complement to $\bfZ$ in $\bfZ'$. 
        Let $f\in \cO_{\bfZ'}(\bfZ')$ such that $f|_\bfA= 1$ and $f|_\bfW= 0$. Take $\bfV$ to be $\bfZ'_f$. 
        \item The general case.\\
        \Rami{
        Embed $\bfZ$ into  a projective variety $\bfZ'$ as an open dense subset. By Case \ref{q.proj.open:3} we can find  an open affine subset $\bfV' \subset \bfZ'$   satisfying $\bfA \subset \bfV'$. 
        Note that $\bfV'\cap \bfZ$ is quasi-affine. The assertion follows now from the previous case.
        }
    \end{enumerate}
\end{proof}
\Rami{
\begin{cor}
 Let a finite group $\Gamma$ act  on a quasi-projective variety $\bfZ$. Then $\bfZ$ can be covered by $\Gamma$-invariant open affine subsets.
\end{cor}
\RamiC{in view of \Cref{lem:fac}, this gives:}
\begin{cor}\label{cor:fac}
 An action of a finite group on a quasi-projective variety is factorizable. 
\end{cor}
}

\subsection{\Rami{Quotients by free actions}}
\begin{lemma}\label{lem:et.des.cols}
    Let $\phi:\bfX\to \bfY$ be a morphism of algebraic varieties. Assume that $\phi_{\RamiC{\bar F}}:\bfX_{\RamiC{\bar F}}\to \bfY_{\RamiC{\bar F}}$ is {\'e}tale. Then so is $\phi$. 
\end{lemma}
\begin{proof}
Without loss of generality we may assume that $\bfX$ and $\bfY$ are affine. 
\begin{enumerate}[Step 1.]
    \item There exists a finite extension $E/F$ such that $\phi_{E}:\bfX_{E}\to \bfY_{E}$ is {\'e}tale.\\
    It is easy to see that for any $E$, $\phi_E$ is flat.
    For any $E/F$ let 
    $$I_E:=\ker(\cO_\bfX(\bfX)\otimes_{\cO_\bfY(\bfY)} \cO_\bfX(\bfX)\to \cO_\bfX(\bfX)).$$ 
    The fact that  $\phi_{E}$ is unramified is equivalent to the fact that $I_E=I_E^2$. The assertion follows now from the fact that $I_F$ is finitely generated (as guaranteed by the Hilbert basis theorem).
    \item $\phi$ is {\'e}tale.\\ Let $E$ be as in the previous step. It is easy to see that the natural map $\bfY_E\to \bfY$ is  finite (and hence integral). The assertion follows now from  descent for {\'e}tale morphisms (see  \cite[Proposition 41.20.6]{SP}).
\end{enumerate}  
\end{proof}
\begin{lemma}\label{lem:etal.bij}
    Let $\phi:\bfZ_1\to \bfZ_2$ be an {\'e}tale map s.t. 
    $\phi(\bar F):\bfZ_1(\bar F)\to \bfZ_2(\bar F)$ is a bijection. Then $\phi$ is an isomorphism.
\end{lemma}
\begin{proof}
    \begin{enumerate}[Step 1.]
        \item Let $\psi:\bfZ_1\to \bfZ_2$ be a standard {\'e}tale map (see \cite[\href{https://stacks.math.columbia.edu/tag/00UB}{Definition 00UB}]{SP})  s.t. 
    $\psi(\bar F):\bfZ_1(\bar F)\to \bfZ_2(\bar F)$ is a injection. Then $\psi$ is an open embedding.\\
    Follows immediately from the definition.
    \item $\phi$ is an isomorphism.\\
    By the previous step we have an open cover $\bfZ_1=\bigcup \bfU_i$ s.t. $\phi|_{\bfU_i}$ is an open embedding. Since $\phi(\bar F)$ is a bijection, we obtain that $$\bfZ_2=\bigcup\phi( \bfU_i).$$ Now we can define $\phi^{-1}$ on each $\phi(\bfU_i)$, and the compatibility follows from the fact that $\phi(\bar F)$ is a bijection.
    \end{enumerate}
\end{proof}

\begin{lemma}\label{lem:Gam}
    Let a finite group $\Gamma$ act  factorizably on a 
    variety 
    $\bfZ$. 
    \Rami{Assume that the action of $\Gamma$ on $\bfZ$ is free (i.e. the action of $\Gamma$ on $\bfZ(\bar F)$ is free)}.    
    Then
    \begin{enumerate}[(i)]
        \item \label{it:etal} The map $\bfZ\to \bfZ//\Gamma$ is {\'e}tale. 
        \item The natural morphism $m:\bfZ\times \Gamma \to \bfZ\times_{\bfZ//\Gamma}\bfZ$ is an isomorphism. 
    \end{enumerate}
\end{lemma}
\begin{proof}
$ $
\begin{enumerate}[(i)]
    \item Follows from \cite[\S II.7]{Mum74} and \Cref{lem:et.des.cols}.
\item $\,$
\begin{enumerate}[{Step} 1.]
    \item The map $m$ it is {\'e}tale.\\
        It is enough to show that $m|_{\bfZ\times \{1\}}$ is {\'e}tale. This map is the diagonal map $\Delta: \bfZ \to \bfZ\times_{\bfZ//\Gamma}\bfZ$, which is {\'e}tale by \eqref{it:etal} \Rami{(as the diagonal of an {\'e}tale map is {\'e}tale -- \cite[\href{https://stacks.math.columbia.edu/tag/02GE}{Lemma 02GE}]{SP})}.        
        \item $m$ induces a bijection on the $\RamiC{\bar F}$-points. \\
        Follows from \Cref{lem:GamBij} since the action of $\Gamma$ is free.
        \item $m$ is an isomorphism.\\
        Follows from the previous steps using \Cref{lem:etal.bij}.
        \end{enumerate}
\end{enumerate}
\end{proof}

\begin{cor}[Galois descent for free actions]\label{cor:GalDes}
    In the setting of the previous lemma, let 
      $Sch_{\bfZ//\Gamma}$ denote the category of schemes over ${\bfZ//\Gamma}$ and $Sch^{\Gamma}_{\bfZ}$ denote the category of schemes over $\bfZ$ equipped with an action of $\Gamma$ which is compatible with the action of $\Gamma$ on $\bfZ$. Consider the functor 
    $\cF:Sch_{\bfZ//\Gamma}\to Sch^{\Gamma}_{\bfZ}$ defined by $\cF(\bfX)=\bfX \times_{\bfZ//\Gamma} \bfZ,$ with $\Gamma$ acting on the second coordinate.  Let $\beta:\cF(\bfX)\to \bfX$ be the projection on the first coordinate.  Then 
    \begin{enumerate}[(i)]
        \item \label{it:ff} $\cF$ is fully faithful. 
        \item \label{it:sh} Given $\bfX \in Sch_{\bfZ//\Gamma}$ and a sheaf 
        $\cV$ on it,
        \RamiB{the pullback $\cV(\bfX)\to (\beta^*\cV)({\cF}(\bfX))$ with respect to
        $\beta$} gives
        an isomorphism 
        $$\cV(\bfX)\cong (\beta^*\cV)(\Dima{\cF}(\bfX))^{\Gamma}.$$
    \end{enumerate}  
\end{cor}
\begin{proof}
$\,$
\begin{enumerate}[(i)]
        \item  \Rami{Let $\bfX_1, \bfX_2\in Sch_{\bfZ//\Gamma}$. We need to show that $\cF$ induces a bijection  $$Mor_{Sch_\bfZ^\Gamma}(\cF(\bfX_1),\cF(\bfX_2))\to Mor_{Sch_{\bfZ//\Gamma}}(\bfX_1,\bfX_2).$$ The previous lemma implies that:
        \begin{equation}\label{cor:GalDes:eq1}
        \text{the maps } \cF(\bfX_i)\to \bfX_i\text{ are {\'e}tale (and surjective)}    
        \end{equation}
        \begin{equation}\label{cor:GalDes:eq2}
        \text{the natural maps } \cF(\bfX_i)\times \Gamma\to \cF(\bfX_i)\times_{\Dima{\bfZ//\Gamma}} \cF(\bfX_i)\text{ are isomorphisms.} \end{equation}
The assertion 
        follows now from} 
        faithfully flat descent for morphisms, see {\it e.g.} \cite[Lecture 9, Theorem 1.1]{Tsimmer}.
        \item  Follows from (\ref{cor:GalDes:eq1},\ref{cor:GalDes:eq2}), using  the fact that a coherent sheaf in the Zariski topology is also a sheaf in the {\'e}tale topology, see {\it e.g.}    \cite[\href{https://stacks.math.columbia.edu/tag/03DT}{Lemma 03DT}]{SP}. 
    \end{enumerate}
\end{proof}

\section{Proof of \Cref{thm:int}}\label{sec:Y}

We will use the following \Rami{straightforward} criterion for sharp integrability:
\begin{lem}\label{lem:crit.int}
    Let $\phi:\tilde \bfZ\to \bfZ$ be a modification. Assume that 
    \begin{enumerate}[(a)]
        \item $\bfZ^{sm}$ is big in $\bfZ$,
        \item $\bfZ^{sm}$  admits an invertible top form $\omega$, and
        \item $\phi^*(\omega)$ can be extended to an invertible top form on $\tilde \bfZ^{sm}$.
    \end{enumerate}
    Then $\phi$ is sharply integrable. 
\end{lem}
We need the following notation:
\begin{notation}
Let $\bfZ$ be a smooth algebraic \Rami{(quasi-projective)} surface and $x\in Z:=\bfZ(F)$.
Define the following:
\begin{itemize}
    \item Let $\mdef{\iota_{\bfZ,n}}: \bfZ^n\to \bfZ^{(n)}$ denote the quotient map.
    \item $\mdef{\Delta^n_\bfZ} \subset \bfZ^n$ the diagonal copy.
    \item $\mdef{\bfZ^{(n)}_{diag}}:={\iota_{\bfZ,n}(\Delta^n_\bfZ)} \subset \bfZ^{\Dima{(n)}}$. Note that it is closed since  $\iota_Z$ is finite.
    \item $\mdef{\bfZ^{[n]}_{diag}}:=\mathfrak  H_{\bfZ,n}^{-1}(\bfZ^{(n)}_{diag}) \subset \bfZ^{\Dima{[n]}}$. 
    \item $\mdef{\bfZ^{(n)}_{x}}:=\iota_{\bfZ,n}(\{(x,\dots,x)\}) \subset \bfZ^{(n)}_{diag}$ and $\mdef{\bfZ^{[n]}_{x}}:=\mathfrak H_{\bfZ,n}^{-1}(\bfZ^{(n)}_{x}) \subset \bfZ^{[n]}_{diag}$.
\end{itemize}
\end{notation}

\begin{prop}[{\cite[7.4.E.3]{BK05}}]\label{prop:codim.hilb}
Let $\bfZ$ be a smooth algebraic \Rami{(quasi-projective)}  surface. 
Then for any $x\in \bfZ(F)$ we have
$\dim \bfZ^{[n]}_{x}=n-1$.
\end{prop}
Together with \Cref{thm:dim.hilb} this proposition gives the following corollary.
\begin{cor}\label{cor:codim.hilb}
Let $\bfZ$ be a smooth algebraic \Rami{(quasi-projective)}  surface. Then 
$$\dim \bfZ^{[n]}-\dim \bfZ^{[n]}_{diag}=n-1.$$
\end{cor}

\begin{notation}
    Write $n=n_1+n_2$. Let $\bfZ$ be a quasi-projective algebraic variety. Denote $$\mdef{\bfZ^{n_1,n_2}}:=\{(z_1,\dots,z_n)\in \bfZ^n| \left\{z_1,\dots,z_{n_1}\} \cap \{z_{n_1+1},\dots,z_{n}\}=\emptyset\right\}.$$
    Denote also $$\mdef{\bfZ^{(n_1,n_2)}}:=\bfZ^{n_1,n_2}//(S_{n_1}\times S_{n_2})$$ and by $$\mdef{\iota_{\bfZ,n_1,n_2}}: \bfZ^{n_1,n_2}\to \bfZ^{(n_1,n_2)}$$  the quotient map.
\end{notation}

\begin{lemma}[See \S \ref{subsec:Pfn1n2} below]\label{lem:n1n2}
\Rami{Let $\bfZ$ be a quasi-projective algebraic variety.} 
    Write $n=n_1+n_2$. Then there exist
    \begin{itemize}
        \item a variety $\bfZ^{[n_1,n_2]}$
        \item morphisms of varieties $\cC^\bfZ_{[n_1,n_2]},\cC^\bfZ_{(n_1,n_2)}, \cC^\bfZ_{n_1,n_2}, \fH_{\bfZ,n_1,n_2}$
        \item an open embedding $\bfZ^{[n_1,n_2]}\subset  \bfZ^{[n_1]}\times \bfZ^{[n_2]}$
    \end{itemize}
    s.t.
    \begin{itemize}
        \item 
    We have the  following commutative diagram:    \begin{equation}\label{=n1n2}
    \begin{tikzcd}[arrows={-Stealth}]
     \bfZ^{[n]} \dar["\mathfrak H_{\bfZ,n}"']   & \bfZ^{[n_1,n_2]}\drar[phantom, "\square"]\dar["\mathfrak H_{\bfZ,n_1,n_2}"']\lar["\cC_{[n_1,n_2]}^\bfZ"'] \rar[phantom,"\sub"] 
  &\bfZ^{[n_1]}\times \bfZ^{[n_2]}\dar["\mathfrak H_{\bfZ,n_1}\times \mathfrak H_{\bfZ,n_2}"]
  \\
     \bfZ^{(n)}    & \bfZ^{(n_1,n_2)}\drar[phantom, "\square"]\lar["\cC_{(n_1,n_2)}^\bfZ"'] \rar[phantom,"\sub"] 
  &\bfZ^{(n_1)}\times \bfZ^{(n_2)}
  \\
     \bfZ^{n}  \uar["\iota_{\bfZ,n}"]
     \arrow[bend right=20,swap]{rr}{\cong}& \bfZ^{n_1,n_2}\uar["\iota_{\bfZ,n_1,n_2}"]  
  \rar[phantom, "\sub"]\lar["\cC_{n_1,n_2}^\bfZ"']   
  &\bfZ^{n_1}\times \bfZ^{n_2}\uar[swap,"\iota_{\bfZ,n_1}\times \iota_{\bfZ,n_2}"] 
\end{tikzcd}
\end{equation}
    \item  The embeddings in the diagram are open.
    \item $\cC_{[n_1,n_2]}^\bfZ$ is {\'e}tale, \item the top left square in the diagram is Cartesian on the level of $\bar F$ points.
    \item The bottom curved arrow is the standard identification $\bfZ^n\cong \bfZ^{n_1}\times \bfZ^{n_2}$.
    \end{itemize}

\end{lemma}
\begin{notation}
    Note that         $\mdef{\bfZ^{[n_1,n_2]}}$,$\mdef{\cC^\bfZ_{[n_1,n_2]}},\mdef{\cC^\bfZ_{(n_1,n_2)}}, \mdef{\cC^\bfZ_{n_1,n_2}}, \mdef{\fH_{\bfZ,n_1,n_2}}$ are defined uniquely by the previous lemma, so we will use these notations in the rest of the section.
\end{notation}
\begin{notation}$\,$
\Rami{Let $\bfZ$ be a quasi-projective algebraic variety. Denote}
    \begin{enumerate}[(i)]
        \item $\mdef{\bfZ^n_0}:=\{(z_1,\dots,z_n)\in \bfZ^n\,|\, \forall i,j, z_i\neq z_j \}.$
        \item $\mdef{\bfZ_0^{(n)}}:=\iota_n(\bfZ_0^n)\subset \bfZ^{(n)}$. By \Cref{lem:GamU}, it is an open subset. 
        \item $\mdef{\bfZ_0^{(n_1,n_2)}}:=\iota_{n_1}\times \iota_{n_2}(\bfZ_0^n)\subset \bfZ^{(n_1,n_2)}$.
    \end{enumerate}
\end{notation}
From \Cref{lem:n1n2} we obtain the following corollary.
\begin{cor}\label{cor:n1n2}
\Rami{Let $\bfZ$ be a quasi-projective algebraic variety.} \Dima{Then we have}
    \begin{enumerate}[(i)]
        \item $\cC_{n_1,n_2}^\bfZ(\bfZ_0^{(n_1,n_2)})=\bfZ_0^{(n)}$.
        \item $\cC_{n_1,n_2}^\bfZ|_{\bfZ_0^{(n_1,n_2)}}$ is an {\'e}tale map.
        \item \label{n1n2:big}$\bfZ_0^{(n)}\sub \bfZ^{(n)}$ and $\bfZ_0^{(n_1,n_2)}\sub \bfZ^{(n_1,n_2)}$ are big subsets.
        \item \label{n1n2:open} The image $\cC_{n_1,n_2}^\bfZ(\bfZ^{(n_1,n_2)}) $ is open.
        \item \label{n1n2:diag} $\bigcup_{n_1=1}^{n-1} \cC_{n_1,n_2}^\bfZ(\bfZ^{(n)}) =\bfZ^{(n)}\smallsetminus \bfZ^{(n)}_{diag}$.
    \end{enumerate}
\end{cor}
The following follows from the Zariski main theorem:
\begin{lem}\label{lem:zar.main}
    Let $\gamma:\bfZ_1\to \bfZ_2$ be a  morphism of algebraic varieties. Assume that:
    \begin{itemize}
        \item $\bfZ_i$ are irreducible.
        \item $\bfZ_2$ normal.
        \item $\gamma$ induces a bijection:  $\bfZ_1(\bar F)\to \bfZ_2(\bar F)$.
        \item \RamiD{$\gamma$ is birational.}
    \end{itemize}
    Then $\gamma$ is an isomorphism.
\end{lem}
\begin{proof}
  \Eitan{Notice that the fibers over geometric points are singletons. \RamiD{Hence $\gamma$ is quasi-finite.} By Zariski main theorem we can decompose: 
    $\gamma=\pi \circ j$ with $j: \bfZ_1 \to \bfZ_3, \pi: \bfZ_3 \to \bfZ_2$ where $\pi$ is finite and $j$ is an open immersion. As $\bfZ_1$ is normal\RamiC{,} and
    finite birational morphism onto a normal variety is an isomorphism\RamiC{,}
    it follows that $\pi$ is an isomorphism. As the image of $j$ must contain all geometric points, it is easy to see that $j$ is also an isomorphism and we are done.}
\end{proof}



\begin{lem}
\label{lem:Hilb2}
    Let $\bfZ=\A^2$. 
    Then there is a commutative diagram:
    \begin{equation}\label{eq:hilb.blow}        
    \begin{tikzcd}[arrows={-Stealth}]      Bl_{\Delta_\bfZ}\bfZ^2\dar["q_{S_2}"]\rar["bl"']   & \bfZ^{2} \dar["\iota_{\bfZ,2}"]
  \\
     \bfZ^{[2]}\rar["\mathfrak H_{\bfZ,2}"']    & \bfZ^{(2)}
\end{tikzcd}
\end{equation}
where the top row is the blowing-up \RamiC{of $\bfZ^2$ along the diagonal $\Delta_\bfZ$,} and the left vertical arrow is the quotient map by the action of $S_2$ given by the flip of the 2 copies of $\bfZ$.
\end{lem}
\begin{proof}
    First let us  construct the map $q_{S_2}$. \Rami{By the definition of the Hilbert scheme $\bfZ^{[2]}$ this means} to construct a scheme $\bfY\subset (Bl_{\Delta \bfZ}\bfZ^2)\times \bfZ^2$ which is finite flat of rank 2 over $Bl_{\Delta \bfZ}\bfZ^2$.
    Realize $Bl_{\Delta\bfZ}\bfZ^2$  as $$\{\RamiC{(l,x,y)}|l \text{ is a line in } \bfZ; x,y\in l\}$$ We get $$(Bl_{\Delta \bfZ}\bfZ^2)\times \bfZ=\{\RamiC{(l,x,y,z)}|l \text{ is a line in } \bfZ;x,y\in l; z\in \bfZ\}$$
    Let $\cI_1,\cI_2,\cI_3$ be the sheaves of ideals in $(Bl_{\Delta \bfZ}\bfZ^2)\times \bfZ$ given by the conditions: 
    \begin{enumerate}[1.]
        \item $x=z $
        \item $y=z $
        \item $z\in l$
    \end{enumerate}
    respectively.
Define $\cI:=\langle \cI_1 \cI_2,\cI_3\rangle$, and let $\bfY$ be its $0$-locus. It is easy to see that $\bfY$ is finite flat  of rank 2 over $Bl_{\Delta \bfZ}\bfZ^2$ and thus defines a map $q_{S_2}:Bl_{\Delta \bfZ}\bfZ^2 \to \bfZ^{[2]}$. By \RamiC{\Cref{cor:fac}} there exists a categorical quotient $Bl_{\Delta \bfZ}\bfZ^2 //S_2$. So  $q_{S_2}$ factors through a map $\gamma:Bl_{\Delta \bfZ}\bfZ^2 //S_2\to \bfZ^{[2]}$. It is easy to see that this map is a bijection on the level of $\bar F$ points \RamiD{and birational}. Also, by \Cref{thm:dim.hilb}  $\bfZ^{[2]}$ is smooth and irreducible. Hence 
\Cref{lem:zar.main} implies that $\gamma$ is an isomorphism.

So we constructed the diagram \eqref{eq:hilb.blow} and proved that $q_{S_2}$ is the quotient map by the action of $S_2$. It is left to show that this diagram is commutative. It is enough to verify it on the level of $\bar F$ points. This follows from the definitions.  
\end{proof}

The following lemma is obvious.
\begin{lemma}\label{lem:BigInt}
Let $\gamma:\bfX\to \bfY$ be a modification of algebraic varieties. Let $\bfU\sub \bfY$ be an open set. Assume that:
\begin{itemize}
    \item $\gamma^{-1}(\bfU)$ is big in $\bfX$
    \item $\gamma|_{\gamma^{-1}(\bfU)}\to \bfU$ is a (sharply) integrable modification.
\end{itemize}
Then $\gamma:\bfX\to \bfY$ is a (sharply) integrable modification.
\end{lemma}

\begin{lem}\label{lem:BigInt2}
    Assume that we have a commutative diagram
       $$\begin{tikzcd}[arrows={-Stealth}]
      \bfZ_{11}\dar["m_1"']   & \bfZ_{12} \dar["m_2"]\lar["e_1"]
  \\
     \bfZ_{21}   & \bfZ_{22}\lar["e_2"'] 
\end{tikzcd}$$
Assume that $e_1$ is onto and {\'e}tale, $e_2$ is of relative dimension zero, $m_i$ are resolutions of singularities, $m_2$ is an integrable modification, and there exists an open smooth set $\bfU \subset \bfZ_{21}$ such that $e_2^{-1}(\bfU)\subset \bfZ_{22}$ is a big subset. Then $m_1$ is an integrable modification.
\end{lem}
\begin{proof}
    Let $\bfV\sub \bfZ_{21}$ be an open set, and $\omega$ be a regular top form on the smooth locus $\bfV^{sm}$ of $\bfV$. We have to show that $m_1^*(\omega)$ can be extended to a top form on $m_1^{-1}(\bfV)$. 
    Since $e_{1}$ is onto {\'e}tale, it is enough to show that $e_{1}^*m_1^*(\omega)$ can be extended to a top form on $e_1^{-1}m_1^{-1}(\bfV)$. Equivalently, it is enough to show that $m_{2}^*e_2^*(\omega)$ can be extended to a top form on $m_2^{-1}e_2^{-1}(\bfV)$. 
    Note that $e_2^*(\omega)$ is regular on $e_2^{-1}(\bfV^{sm}\cap \bfU)=e_2^{-1}(\bfV)\cap e_2^{-1}(\bfU)$. Since $e_2^{-1}(\bfU)$ is big in $\bfZ_{22}$ this implies that $e_2^*(\omega)$ can be extended to a regular form on the smooth locus of $e_2^{-1}(\bfV)$. Now, the fact that $m_2$ is integrable implies the assertion. 
\end{proof}

\begin{proof}[Proof of \Cref{thm:int}]
        Let $\bfZ:=\bA^2$.
   By \Cref{cor:Hil.chow.res} the Hilbert-Chow map is a  resolution of singularities. So we need to show that it is  integrable. We  will do it by analyzing the following cases.
    
    \begin{enumerate}[{Case} 1.]    
\item $n=2$.\\
    Let $\bfU \sub \bfZ^2$ be the complement to the diagonal. 
    By \Cref{lem:GamU}  $\iota_{\bfZ,2}(\bfU)$ is open in $\bfZ^{(2)}$, and  $\iota_{\bfZ,2}(\bfU)  \cong \bfU//S_2$. By \Cref{lem:Gam}, $\bfU//S_2$ is smooth. We obtain that $\bfV:=\iota_{\bfZ,2}(\bfU)$ is an open subset of the smooth locus of $\bfZ^{(2)}$. In fact, it is equal to this smooth locus.
    Also, $\iota^{-1}_{\bfZ,2}(\bfV)\subset \bfZ^2$ is big. 
    Let $\omega_{\bfZ^2}$ be the standard top form on $\bfZ^2=\bA^4$. It is easy to see that it is $S_2$-invariant. Thus, by \Cref{cor:GalDes}, there exists a top form $\omega_\bfV$ s.t. $\iota_{\bfZ,2}^*(\omega_\bfV)=\omega_{\bfZ^2}|_{\bfU}$. It is easy to see that $\omega_\bfV$ is an invertible form.     
    Consider $\omega_\bfV$ as a rational top form on $\bfZ^{(2)}$.
    
    \Rami{By \Cref{lem:crit.int}} it is enough to show that
$\mathfrak H_{\bfZ,2}^*(\omega_\bfV)$ is an invertible top form on $\bfZ^{\Rami{[2]}}$.

By \Cref{lem:Hilb2} we have the following commutative diagram:
   $$        
    \begin{tikzcd}[arrows={-Stealth}] Bl_{\Delta_\bfZ}\bfZ^2\dar["q_{S_2}"]\rar["bl"']   & \bfZ^{2} \dar["\iota_{\bfZ,2}"]
  \\     \bfZ^{[2]}\rar["\mathfrak H_{\bfZ,2}"']    & \bfZ^{(2)}
\end{tikzcd}
$$
where the top row is the blowing-up and the left vertical arrow is the quotient map by the action of $S_2$.

The assertion follows now from the following 2 statements:
\begin{enumerate}
\item\label{thm:Hil.Chow.int:stat.a}  For an invertible form $\omega$ on $\bfZ^2$, the zero locus of the form $bl^*(\omega)$ is the divisor $bl^{-1}(\Delta_\bfZ)$ with multiplicity $1$.
\item\label{thm:Hil.Chow.int:stat.b} If $\omega$ is a rational form on $\bfZ^{[2]}\cong Bl_{\Delta_\bfZ}\bfZ^2//S_2$ s.t. $q_{S_2}^*(\omega)$ is a regular form and its zero locus is the divisor $bl^{-1}(\Delta_\bfZ)$ with multiplicity $1$ then $\omega$ is regular.
\end{enumerate}
\begin{enumerate}[Proof of (a):]
    \item This is a standard property of a  blowing up of a smooth variety along smooth subvariety of co-dimension $2$.
    \item As in the proof of \Cref{lem:Hilb2} we can realize $Bl_{\Delta_\bfZ}\bfZ^2$ as 
    $\{\RamiC{(l,x,y)}|l \text{ is a line in } \bfZ; x,y\in l\}.$ This gives a map $Bl_{\Delta_\bfZ} \bfZ^2\to \cL$, where $\cL$ is the collection of lines in $\bfZ=\bA^2$. This map is $S_2$-invariant, so we get a commutative diagram:       \begin{equation}\label{diag:bl.hilb.L} 
    \begin{tikzcd}[arrows={-Stealth}]  Bl_{\Delta_\bfZ}\bfZ^2\dar["q_{S_2}"]\drar[""']   &
  \\     (Bl_{\Delta_\bfZ}\bfZ^2)//S_2\rar[""']    & \cL
\end{tikzcd}
\end{equation}   

The statement is Zariski local on $\cL$. Let $\tilde\cL=\bigsqcup \cL_i\to \cL$ be a   Zariski cover that trivializes the tautological bundle.
When we pull the diagram \eqref{diag:bl.hilb.L} to $\tilde L$ we obtain a diagram isomorphic to: 
\begin{equation}
    \begin{tikzcd}[arrows={-Stealth}]      \tilde\cL\times \bA^2 \dar[""]\drar[""']   &
  \\
     \tilde\cL\times \bA^2//S_2 \rar[""']    & \tilde\cL
\end{tikzcd}
\end{equation}   
where $S_2$ acts on $\bA^2$ by flipping the coordinates.

Let $q'_{S_2}:\bA^2\to  \bA^2//S_2$ be the quotient map.

It is enough to show that 
if $\eta$ is a rational form on $\bA^2//S_2$ s.t. $(q'_{S_2})^*(\eta)$ is regular and its zero locus is  he divisor $\Delta_{\bA^1}\subset \bA^2$ with multiplicity $1$ then \Rami{$\eta$} is regular. This is a straightforward computation.
\end{enumerate}

    \item The general case.
    We prove the statement by induction on $n$. Case $n=1$ is obvious. Case $n=2$ is the previous case.  Assume $n>2$. 
    By \Cref{prop:codim.hilb} we have $\dim  \bfZ^{(n)} -\dim  \bfZ^{(n)}_{diag}>1$. Thus by \Cref{lem:BigInt} it is enough to prove that  $$\mathfrak H_{\bfZ,n}|_{\bfZ^{(n)}}:\bfZ^{(n)} \smallsetminus \bfZ^{(n)}_{diag} \to \bfZ^{[n]} \smallsetminus \bfZ^{[n]}_{diag}$$
     is  an integrable modification. 
     Write $n=n_1+n_2$ with $n_1,n_2>0$.
    Denote $\bfU:=\cC_{n_1,n_2}^\bfZ(\bfZ^{(n_1,n_2)})\subset \bfZ^{(n)}$. By \Cref{cor:n1n2}\eqref{n1n2:open}, it is an open subscheme. Denote $\bfV:=\mathfrak H_{\bfZ,n}^{-1}(\bfU)$.
    By \Cref{cor:n1n2}\eqref{n1n2:diag}, it is enough to show that 
     $$\mathfrak H_{\bfZ,n}|_{\bfV}:\bfV \to \bfU$$ is an integrable modification, for any decomposition $n=n_1+n_2$. 
     Let $\mathfrak H_{\bfZ,n_1,n_2}:\bfZ^{[n_1,n_2]}\to \bfZ^{(n_1,n_2)}$ be as in \Cref{lem:n1n2}.
     \Cref{lem:n1n2} and the induction hypothesis imply that 
     $\mathfrak H_{\bfZ,n_1,n_2}$ is an integrable modification. 
     By \Cref{lem:n1n2}, the following diagram is commutative.
     $$\begin{tikzcd}[arrows={-Stealth}]
     \bfV \dar["\overset{o}{\mathfrak H}_{\bfZ,n}"']   & \bfZ^{[n_1,n_2]}\dar["\mathfrak H_{\bfZ,n_1,n_2}"']\lar["\overset{\circ}{\cC}_{[n_1,n_2]}"']  \\
     \bfU    & \bfZ^{(n_1,n_2)} \lar["\overset{\circ}{\cC}_{(n_1,n_2)}"'] 
  \end{tikzcd},$$
     where the maps $\overset{\circ}{\cC}_{[n_1,n_2]}$, $\overset{\circ}{\cC}_{(n_1,n_2)}$, and $\overset{o}{\mathfrak H}_{\bfZ,n}$ are obtained by restriction from the maps $\cC_{[n_1,n_2]}$ and ${\mathfrak H}_{\bfZ,n}$.
     Moreover, this diagram is a Cartesian square on the level of $\RamiC{\bar F}$-points. This implies that the map $\overset{\circ}{\cC}_{[n_1,n_2]}$ is onto. By \Cref{lem:n1n2}, it is also {\'e}tale. By \Cref{cor:n1n2}\eqref{n1n2:big} and \Cref{lem:BigInt2}, 
     $\mathfrak H_{\bfZ,n}|_{\bfV}:\bfV \to \bfU$ is an integrable modification, as required.
\end{enumerate}

\end{proof}

\subsection{Proof of \Cref{lem:n1n2}}\label{subsec:Pfn1n2}

Note that by \Cref{lem:GamU}, we have the Cartesian square 
\begin{equation}\label{eq:cart.Z(n1,n2)}    
\begin{tikzcd}
     \bfZ^{(n_1,n_2)}\drar[phantom, "\square"]\rar[phantom,"\sub"] 
  &\bfZ^{(n_1)}\times \bfZ^{(n_2)}
  \\
 \bfZ^{n_1,n_2}\uar["\iota_{\bfZ,n_1,n_2}"]  
  \rar[phantom, "\sub"]
  &\bfZ^{n_1}\times \bfZ^{n_2}\uar[swap,"\iota_{\bfZ,n_1}\times \iota_{\bfZ,n_2}"] 
\end{tikzcd}
\end{equation}

Denote $\bfZ^{[n_1,n_2]}:=(\mathfrak H_{\bfZ,n_1}\times \mathfrak H_{\bfZ,n_2})^{-1}(\bfZ^{(n_1,n_2)})$.
This gives us the Cartesian square
\begin{equation}   
\label{eq:cart.Z[n1,n2]}
\begin{tikzcd}[arrows={-Stealth}]\bfZ^{[n_1,n_2]}\drar[phantom, "\square"]\dar["\mathfrak H_{\bfZ,n_1,n_2}"']\rar[phantom,"\sub"] 
  &\bfZ^{[n_1]}\times \bfZ^{[n_2]}\dar["\mathfrak H_{\bfZ,n_1}\times \mathfrak H_{\bfZ,n_2}"]
  \\
     \bfZ^{(n_1,n_2)}\rar[phantom,"\sub"] 
  &\bfZ^{(n_1)}\times \bfZ^{(n_2)}
  \end{tikzcd}
  \end{equation}

\begin{definition}
    For an algebraic variety $\bfZ$ define the subfunctor $$\mdef{Hilb_{n_1,n_2}}(\bfZ)\sub Hilb_{n_1}(\bfZ)\times Hilb_{n_2}(\bfZ): Sch_F^{op}\to sets$$ by 
    \begin{multline*}
        Hilb_{n_1,n_2}(\bfZ)(\bfS):=\\
        \{(\bfY_1,\bfY_2)\in Hilb_{n_1}(\bfZ)(\bfS)\times Hilb_{n_2}(\bfZ)(\bfS) \, \vert \, \bfY_1\cap \bfY_2=\emptyset\}
    \end{multline*}
\end{definition}

\begin{lemma}\label{lem:subfunct}
    The subfunctor $Hilb_{n_1,n_2}$ is represented by the open subscheme $\bfZ^{[n_1,n_2]}$.
\end{lemma}
For the proof we will need the following straightforward lemma:
\begin{lem}\label{lem:cancel.set}
Consider the following commutative diagram in arbitrary category.
$$
\begin{tikzcd}
    Z_{11} \arrow[d,"\delta_{1}"] \arrow[r, "\gamma_{11}"] & Z_{12}\arrow[dr, phantom, "\square"] \arrow[d,"\delta_{2}"] \arrow[r,"\gamma_{12}"]  & Z_{13} \arrow[d,"\delta_{3}"] \\
    Z_{21} \arrow[r,"\gamma_{21}"]  & Z_{22} \arrow[r, "\gamma_{22}"]  & Z_{23}
\end{tikzcd}
$$
Assume also that we have:
$$
\begin{tikzcd}
    Z_{11} \arrow[d,"\delta_{1}"] \arrow[r, "\gamma_{12}\circ\gamma_{11}"]\arrow[dr, phantom, "\square"] & Z_{13} \arrow[d,"\delta_{3}"]\\
    Z_{21} \arrow[r,"\gamma_{22}\circ\gamma_{21}"]  & Z_{23} 
\end{tikzcd}
$$
Then we have:
$$
\begin{tikzcd}
    Z_{11} \arrow[d,"\delta_{1}"] \arrow[r,"\gamma_{11}"] \arrow[dr, phantom, "\square"] & Z_{12} \arrow[d,"\delta_{2}"] \\
     Z_{21} \arrow[r, "\gamma_{21}"]  & Z_{22}
\end{tikzcd}
$$
\end{lem}

\begin{proof}[Proof of \Cref{lem:subfunct}]
$ $
\begin{enumerate}[Step 1.]
    \item 
$Hilb_{n_1,n_2}$ is represented by an open subscheme  of $\bfZ^{[n_1]}\times \bfZ^{[n_2]}$.\\
\begin{enumerate}[Step a.]
\item Construction of $\bfZ^{[n_1,n_2]}_0$.\\
Let $\tilde \bfZ^{[n_i]}\subset \bfZ\times \bfZ^{[n_i]}$ be the tautological scheme over $\bfZ^{[n_i]}$, i.e. the subscheme of $\bfZ\times \bfZ^{[n_i]}$ that corresponds to the identity map under the isomorphism 
$$Mor(\bfZ^{[n_i]},\bfZ^{[n_i]})\cong Hilb_{[n_i]}(\bfZ)(\bfZ^{[n_i]}).$$
Let  $\cW:=\tilde \bfZ^{[n_1]} \times_\bfZ \tilde\bfZ^{[n_2]}$. We have a natural embedding $\cW\subset \bfZ \times \bfZ^{[n_1]} \times \bfZ^{[n_2]}$.
Let $pr: \cW\to \bfZ^{[n_1]}\times \bfZ^{[n_2]}$ be the projection. Note that it is finite. Let $\bfZ^{[n_1,n_2]}_0:= \bfZ^{[n_1]}\times \bfZ^{[n_2]} \smallsetminus pr(\cW),$ and consider it as an open subscheme of $ \bfZ^{[n_1]}\times \bfZ^{[n_2]}$. 
\item Proof that $\bfZ^{[n_1,n_2]}_0$
 represents the functor $Hilb_{n_1,n_2}(\bfZ)$.\\
Note that for any $S\in Sch_F$ we have 
\begin{equation}    \label{eq:mor.Z0}
Mor(\bfS, \bfZ^{[n_1,n_2]}_0)=\left\{ \gamma \in Mor(\bfS,\bfZ^{[n_1]}\times \bfZ^{[n_2]})| \gamma^*(\cW)=\emptyset\right \},
\end{equation}
where $\gamma^*(\cW)$ is the object that makes the following square Cartesian:
$$\begin{tikzcd}
    \gamma^*(\cW) \drar[phantom, "\square"] \rar[""]\dar[""] &\cW\dar["pr"]\\
    \bfS \rar["\gamma"]&\bfZ^{[n_1]}\times \bfZ^{[n_2]} .
\end{tikzcd}$$
Write $\gamma=(\gamma_1,\gamma_2)$. The maps $\gamma_i$ corresponds to a $\bfY_i\in Hilb_{n_i}(\bfZ)(\bfS)$. We will show that $\gamma^*(\cW)=\bfY_1\cap \bfY_2$.

We have Cartesian squares:
\begin{equation}\label{eq:cart.sq.hilb}
\begin{tikzcd}
    \bfY_i \drar[phantom, "\square"] \rar[""]\dar[""] &\tilde \bfZ^{[n_i]}\dar["pr"]\\
    \bfS \rar["\gamma"]&
    \bfZ^{[n_i]} .
\end{tikzcd}
\end{equation}

Consider the following diagram:
$$
\begin{tikzcd}
    \bfY_1\cap \bfY_2
    \drar[phantom, "1"] \arrow[d] \arrow[r] & \cW \drar[phantom, "2"] \arrow[d,"pr"] \arrow[r,""] & \bfZ \arrow[d,"diag"] \\
    \bfY_1\times_\bfS \bfY_2 \arrow[d, ""] \arrow[r] \drar[phantom, "3"]& \tilde \bfZ^{[n_1]}\times \tilde \bfZ^{[n_2]} \arrow[d,""] \arrow[r,""] & \bfZ\times \bfZ \\
    \bfS \arrow[r,"\gamma"] & \bfZ^{[n_1]}\times  \bfZ^{[n_2]}
\end{tikzcd}
$$
where $diag$ is the diagonal map.
The square 3 is Cartesian because of \eqref{eq:cart.sq.hilb}.  The square 2 is Cartesian by the definition of $\cW$. It is obvious that the square 
$$
\begin{tikzcd}
    \bfY_1\cap \bfY_2
    \drar[phantom, "1\circ 2"] \arrow[d] \arrow[r] \arrow[d,"pr"] \arrow[r,""] & \bfZ \arrow[d,"diag"] \\    \bfY_1\times_\bfS \bfY_2 \arrow[r]  & \bfZ\times \bfZ 
\end{tikzcd}
$$
is Cartesian. Therefore, by \Cref{lem:cancel.set} the square 1 is Cartesian and thus 
$\gamma^*(\cW)=\bfY_1\cap \bfY_2$
Together with \eqref{eq:mor.Z0} this completes the step.
\end{enumerate}
\item $Mor(\spec \bar F,\bfZ^{[n_1,n_2]})=Hilb_{n_1,n_2}(\bfZ)(\spec \bar{F}).$\\
Follows directly from the definitions of $\bfZ^{[n_1,n_2]}$ and $Hilb_{n_1,n_2}(\bfZ)$, the characterization of the Hilbert-Chow map (\Cref{thm:charact-HilChow}) and \Cref{lem:GamBij}.
\item End of the proof.\\
Follows from the 2 previous steps and the fact that an open subset in a variety is determined by its $\bar F$ points. The later follows from the Nullstellensatz. 
\end{enumerate}
\end{proof}

\begin{notation}
    Define a natural transformation $${\RamiC{\cC}_{[n_1,n_2]}^\bfZ}:Hilb_{n_1,n_2}(\bfZ)\to Hilb_{n}(\bfZ)$$
    by 
    $$\mdef{\RamiC{\cC}_{[n_1,n_2]}^\bfZ(\bfY_1,\bfY_2)}:=\bfY_1
    \sqcup \bfY_2$$
   By \Cref{lem:subfunct} this morphism defines a morphism $\bfZ^{[n_1,n_2]}\to \bfZ^{[n]},$ that we will also  denote  by $\RamiC{\cC}_{[n_1,n_2]}^\bfZ$. 
\end{notation}

\begin{notation}
    Let $\overline{\cC}^{\bfZ}_{(n_1,n_2)}:\bfZ^n//(S_{n_1}\times S_{n_2})\to \bfZ^n//S_n=\bfZ^{(n)}$ denote the natural map. By \Cref{lem:GamU}, $\bfZ^{(n_1,n_2)}$ can be considered as an open subset of $\bfZ^n//(S_{n_1}\times S_{n_2})$. Denote the restriction of $\overline{\cC}^{\bfZ}_{(n_1,n_2)}$ to $\bfZ^{(n_1,n_2)}$ by $\cC^{\bfZ}_{(n_1,n_2)}$. Finally, we let $\cC^{\bfZ}_{n_1,n_2}:\bfZ^{n_1,n_2}\to \bfZ^n$ be the map given by concatenation of tuples.
\end{notation}
Now we defined all the arrows in the diagram \eqref{=n1n2}. It remains to show that:

\begin{enumerate}[(a)]
    \item \label{Pfn1n2:com}The diagram is commutative.
    \item \label{Pfn1n2:Car} The top left square is Cartesian on the level of $\RamiC{\bar F}$ points.
    \item \label{Pfn1n2:etale}The map $\cC^{\bfZ}_{[n_1,n_2]}$ is {\'e}tale. 
\end{enumerate}
It is enough to prove \eqref{Pfn1n2:com} on the level of $\RamiC{\bar F}$ points. Thus \eqref{Pfn1n2:com} and \eqref{Pfn1n2:Car} are straightforward computations in view of \Cref{lem:GamU}, the definition of Hilbert scheme (\Cref{def:Hscheme}), and the characterization of the Hilbert-Chow map (\Cref{thm:charact-HilChow}).

It remains to show \eqref{Pfn1n2:etale}. 
By \Cref{thm:dim.hilb}, the variety $\bfZ^{[n]}$ is smooth. By diagrams (\ref{eq:cart.Z(n1,n2)},\ref{eq:cart.Z[n1,n2]}) the variety   $\bfZ^{[n_1,n_2]}$ is an open subset of $\bfZ^{[n]}$, and thus  is also smooth. 
It is enough to show that the map $\cC^{\bfZ}_{[n_1,n_2]}$ is {\'e}tale at any closed point of 
$\bfZ^{[n_1,n_2]}$. Equivalently, we have to show that for any finite field extension $E$ over $F$, and any $y\in \bfZ^{[n_1,n_2]}(E)$, the differential $d_y\cC^{\bfZ}_{[n_1,n_2]}$ is an isomorphism. Without loss of generality we assume that $E=F$ and $\bfZ$ is affine. 

To $x\in \bfZ^{[n]}(F)$ we can assign an ideal $I\vartriangleleft \cO_{\bfZ}(\bfZ)$. This gives us an identification 
$$T_x\bfZ^{[n]}=\{\widetilde{I} \vartriangleleft \cO_{\bfZ}(\bfZ)[t]/t^2 \,\vert \, (\cO_{\bfZ}(\bfZ)[t]/t^2)/\widetilde{I}\simeq (F[t]/t^2)^n \text{ and }\widetilde{I}/t=I\}$$
Here, the isomorphism is an isomorphism of $F[t]/t^2$-modules. 
Define 
$$\gamma_{\bfZ,n,x}:\Hom(I,\cO_{\bfZ}(\bfZ)/I)\to T_x\bfZ^{[n]}$$
by $\gamma_{\Rami{\bfZ,n,x}}(\eps)=\{a+tb\, \vert \, a\in I, b\in \eps(a)\}$. It is easy to see that $\gamma_{\Rami{\bfZ,n,x}}$ is an isomorphism (cf. \cite[Lemma 7.2.5]{BK05}). 
Let $$y=(x_1,x_2)\in \bfZ^{[n_1,n_2]}(F)\subset \bfZ^{[n_1]}(F)\times \bfZ^{[n_2]}(F).$$ 
\Rami{We have to show that $d_y\cC^{\bfZ}_{[n_1,n_2]}$ is an isomorphism.}
Let $I_1,I_2\vartriangleleft \cO_{\bfZ}(\bfZ)$ be the ideals corresponding to the points $x_1,x_2$. The Chinese remainder theorem gives an identification 
$$\delta_0:\cO_{\bfZ}(\bfZ)/I_1\oplus \cO_{\bfZ}(\bfZ)/I_2\cong \cO_{\bfZ}(\bfZ)/(I_1\cap I_2)$$ 
This gives an identification
$$\delta_1:\Hom(I_1\cap I_2,\cO_{\bfZ}(\bfZ)/I_1\oplus \cO_{\bfZ}(\bfZ)/I_2)\cong \Hom(I_1\cap I_2,\cO_{\bfZ}(\bfZ)/(I_1\cap I_2))$$
Define a morphism $$\delta: \Hom(I_1,\cO_{\bfZ}(\bfZ)/I_1) \times \Hom(I_2,\cO_{\bfZ}(\bfZ)/I_2)\to \Hom(I_1\cap I_2,\cO_{\bfZ}(\bfZ)/(I_1\cap I_2))$$  by 
$$\delta(\eps_1,\eps_2):= \delta_1(\eps_1|_{I_1\cap I_2},\eps_2|_{I_1\cap I_2})$$
It is easy to see that the following diagram is commutative.

$$
\hspace{-1.5em}
\begin{tikzcd}
\mathbf{T}_{x_1}\bfZ^{[n_1]} \times \mathbf{T}_{x_2}\bfZ^{[n_2]}
  \arrow[r, "\cong", phantom]
&
\mathbf{T}_y\bfZ^{[n_1,n_2]}
  \arrow[r,"d_y\cC^{\bfZ}_{[n_1,n_2]}"]
&
\mathbf{T}_{\cC^{\bfZ}_{[n_1,n_2]}(y)}\bfZ^{[n]}
\\
\Hom(I_1,\cO_{\bfZ}(\bfZ)/I_1) \times \Hom(I_2,\cO_{\bfZ}(\bfZ)/I_2)
  \arrow[
    u,
    "\gamma_{\bfZ,n_1,x_1} \times \gamma_{\bfZ,n_2,x_2}",
    description,
    "\rotatebox{90}{$\sim$}"'
  ]
  \arrow[rr,"\delta"]
& &
\Hom(I_1\cap I_2,\cO_{\bfZ}(\bfZ)/(I_1\cap I_2))
  \arrow[
    u,
    "\gamma_{\bfZ,n,\cC^{\bfZ}_{[n_1,n_2]}(y)}",
    description,
    "\rotatebox{90}{$\sim$}"'
  ]
\end{tikzcd}
$$
Thus it is enough to show that $\delta$ is an isomorphism. 
Let $\cI_1,\cI_2\vartriangleleft \cO_{\bfZ}$ be sheaves of ideals corresponding to the ideals $I_1,I_2$.  $\delta$ defines a morphism of sheaves 
$$\widetilde{\delta}: \cH om(\cI_1,\cO_{\bfZ}(\bfZ)/\cI_1) \times \cH om(\cI_2,\cO_{\bfZ}(\bfZ)/\cI_2)\to \cH om(\cI_1\cap \cI_2,\cO_{\bfZ}(\bfZ)/(\cI_1\cap \cI_2)),$$
where $\cH om$ denotes internal $\Hom$ of sheaves. 
It is enough to show that $\widetilde{\delta}$ is an isomorphism. Let $U_i$ be the complement to the zero locus of $\cI$. It is enough to show that $\delta|_{U_i}$ is an isomorphism for $i=1,2$. This is obvious. 
\section{Proof of \Cref{thm:main}}\label{sec:Pfmain}
\Cref{thm:main} follows from \Cref{thm:int} and the following lemma:
\begin{lemma}
    $((\bA^2)^{(n)})^{sm}$ admits an invertible top form.
\end{lemma}
\begin{proof}
    Let $\bfZ=\bA^2$. Let $(\bfZ^{n})^0\subset \bfZ^{n}$ be the open set of tuples of pairwise different points in $\bfZ$. By \Cref{lem:GamU} we have the following Cartesian square:
$$\begin{tikzcd}
    (\bfZ^n)^0 \drar[phantom, "\square"] \rar[phantom,"\subset"]\dar["\iota_{\bfZ,n}^0"] &\bfZ^n\dar["\iota_{\bfZ,n}"]\\
    (\bfZ^{(n)})^0 \rar[phantom,"\subset"]&
    \bfZ^{(n)}
\end{tikzcd}$$
with the horizontal inclusions being open.
    Let 
    $\omega_{\bfZ^{n}}$ 
    be the standard form on $\bfZ^{n}$ and let 
    $\omega_{(\bfZ^{n})^0}$ be its restriction to $(\bfZ^{n})^0$. 
    By \Cref{lem:Gam} the map $\iota_{\bfZ,n}^0$ is {\'e}tale.  So, 
    $\Omega^{top}((\bfZ^{n})^0)\cong (\iota_{\bfZ,n}^0)^*(\Omega^{top}((\bfZ^{(n)})^0))$

        Note that $\omega_{(\bfZ^{n})^0}$ is $S_n$ invariant. So by \Cref{cor:GalDes}
     it gives a top form $\omega_{(\bfZ^{(n)})^0}$ on $(\bfZ^{(n)})^0$ s.t. $(\iota_{\bfZ,n}^0)^*(\omega_{(\bfZ^{(n)})^0})=\omega_{(\bfZ^{n})^0}$. 
    
    Since $\iota_{\bfZ,n}^0$ is {\'e}tale, the fact that  $\omega_{(\bfZ^{\Dima{n}})^0}$  is invertible implies that      $\omega_{(\bfZ^{(n)})^0}$ is invertible. It is easy to see that $(\bfZ^{(n)})^0$ is big in $\bfZ^{(n)}$.  Thus $\omega_{(\bfZ^{(n)})^0}$ can be extended to an invertible top form on $(\bfZ^{(n)})^{sm}$ as required.
\end{proof}
\begingroup
  \let\clearpage\relax
  \let\cleardoublepage\relax 
  \printindex
\endgroup

\bibliographystyle{alpha}
\bibliography{Ramibib}

\end{document}